\documentclass{article}
\usepackage{amsmath}
\usepackage{amsthm}
\usepackage{amssymb}
\usepackage[dvips]{graphicx}
\usepackage{latexsym}
\usepackage{mathtools}
\usepackage{geometry}
 \renewcommand{\equation}

\newtheorem{prop}{Proposition}
\newtheorem{lem}{Lemma}
\newtheorem{thm}{Theorem}
\newtheorem{cor}{Corollary}
\newtheorem{fact}{Fact}
\newtheorem{defini}{Definition}

\title{Lens spaces which are realizable as closures of homology cobordisms over planar surfaces}
\author{Nozomu Sekino}
\date{}
\begin{document}
\maketitle

\begin{abstract}
We determine the condition on a given lens space having a realization as a closure of homology cobordism over a planar surface with a given number of boundary components. As a corollary, we see that every lens space is represented as a closure of homology cobordism over a planar surface with three boundary components. In the proof of this corollary, we use Chebotarev density theorem.
\end{abstract}

\section{Introduction} \label{sec1}
It is known that every connected orientable 3-manifold $M$ has a fibered link $L$. 
This is equivalent to that $M$ admits a surface $F$ bounding a link $L$ such that $M \setminus (N(L) \cup N(F))$ is homeomorphic to $F \times [0,1]$, where $N(\cdot)$ stands for a  (sufficiently small) open tubular neighborhood. 
However, finding fibered links in a given 3-manifold is difficult in general. 
Since existence of a fiber surface of fixed type in a manifold concerns with the structure of the manifold, there are many works about these. 
About fibered links with planar fiber surface in lens spaces, 
lens space of type $(p,q)$ has such a fibered link whose number of components is as the same number as one plus the length of a continued fraction expansion of $\frac{p}{q}$. 
However it is unknown that there exists a universal upper bound for the number of boundary components of planar fibered surface for every lens space. \\
There is an analogue of fibered links, $\it{homologically \ fibered \ links}$ \cite{homfib}; 
A link $L$ in a 3-manifold $M$ is called a $\it{homologically \ fibered \ link}$ 
if there exists a Seifert surface $F$ of $L$ such that $M \setminus (N(L) \cup N(F))$ is homeomorphic as sutured manifold to a homology cobordism over the surface. 
This is equivalent to that $M$ is obtained by pasting the upper boundary and the lower boundary of a homology cobordism over the surface. 
We may get refinements of notions concerning with surface cross intervals for homology cobordisms. (See \cite{hc} for example.) 
Recently, Nozaki \cite{nozaki} showed that every lens space contains a genus one homologically fibered knot by solving some algebraic equations. 
In this paper, like \cite{nozaki}, we give the relation between the existence of homologically fibered link of a given number of components whose Seifert surface is a planar surface in a given lens space and some algebraic equation as follows. 
We denote by $L(p,q)$ a lens space of type $(p,q)$, and $\Sigma_{g,n}$ a surface of genus $g$ with $n$ boundary components.\\

 \begin{prop} \label{prop}
  $L(p,q)$ has a realization as a closure of a homology cobordism over $\Sigma_{0,n+1}$ ($n\geq 1$)
if and only if there exist integers $a_{h}$, $l_{i,j}$ and $t_k$($h,i,j,k=1,2,\dots,n$, $i\lneq j$) satisfying following equation: \\

\begin{eqnarray}
\left|
\begin{array}{ccccc}
   p      &   -qa_{1}  &   -qa_{2}   & \cdots &   -qa_{n}\\
   a_{1}  &     t_{1}    &   l_{1,2}    &  \cdots &   l_{1,n} \\
   a_{2}  &    l_{1,2}   &    t_{2}     & \cdots  &   l_{2,n}\\
\vdots &  \vdots    &   \vdots  &  \ddots & \vdots \\
   a_{n}  &   l_{1,n}    &     l_{2,n}  &  \cdots &     t_{n}\\
\end{array}
\right| = \pm 1 
\end{eqnarray}

\end{prop}

Note that since any manifold which has a realization as a closure of a homology cobordism over $\Sigma_{0,1}$ is a homology 3-sphere, no lens spaces have such a realization. 
Note also that if we have a solution of $(1)$ for some $n$, by setting $a_{n+1}=l_{i,n+1}=0$ for all $i$ and $t_{n+1}=1$ we get that for $n+1$. 
Hence a lens space which has a realization as a closure of homology cobordism over $\Sigma_{0,n+1}$ has that of $\Sigma_{0,n+2}$, too. 
Because of this reason, it seems meaningful to find the minimal $n$ of $(1)$ for a given lens space. 
However, this question is completely answered as follows:

\vspace{1.0cm}
 \begin{thm} \label{thm1}
 $L(p,q)$ has a realization as a closure of a homology cobordism over $\Sigma_{0,2}$ 
if and only if $q$ or $-q$ is a quadratic residue modulo $p$.
 \end{thm}

 \begin{thm} \label{thm2}
 Every lens space has a realization as a closure of a homology cobordism over $\Sigma_{0,3}$.
 \end{thm}

The remainder of this paper is organized as follows. 
In Section \ref{sec2}, we give preliminary terminologies. 
In Section \ref{sec3}, we prove Proposition \ref{prop}. 
Section \ref{sec4} is devoted to the proofs for theorems.
Finally, we give some applications for the invariant hc($\cdot$) defined in \cite{hc} in Section \ref{sec5}.

\subsection*{Acknowledgements}
The author would like to give his thanks to Prof. Takuya Sakasai for giving many meaningful
advices, and Prof. Yuta Nozaki for telling algebraic operations to the author.

\section{Preliminary} \label{sec2}
 \subsection{Homology cobordisms  (Section $2.4$ of \cite{homcob})}
A $\it {homology \ cobordism \ over \ \Sigma_{g,n}}$ ($n\geq 1$) is a triad ($X,\partial_{+}X,\partial_{-}X$), 
where $X$ is an oriented compact 3-manifold and $\partial_{+}X \cup \partial_{-}X$ is a partition of $\partial X$, and $\partial_{\pm} X$ are homeomorphic to $\Sigma_{g,n}$ satisfying:\\

\begin{itemize}
  \item $\partial_{+}X \cup \partial_{-}X =  \partial X$.
  \item $\partial_{+}X \cap \partial_{-}X =  \partial(\partial_{+}X)$.
  \item The induced maps $(i_{\pm})_{*}: H_{*}(\partial_{\pm} X) \rightarrow H_{*}(X)$ are isomorphisms, where $i_{\pm}: \partial_{\pm}X \rightarrow X$ is the inclusions.\\
\end{itemize}

Note that the third condition is equivalent to the condition that $X$ is connected and $i_{\pm}$ induce isomorphisms on $H_{1}(\cdot)$. 
Moreover by using the Poincar$\rm{\acute{e}}$ duality, we see that inducing isomorphism of one of $(i_{+})_{*}$ and $(i_{-})_{*}$ is sufficient. \\
By pasting $\partial_{+}X$ and $\partial_{-}X$ using any boundary-fixing homeomorphism, we get a closed 3-manifold $M$ and a surface $F$ in $M$, which is the image of $\partial_{+}X$. In this situation, we say $M$ is representable as a closure of a homology cobordism over $\Sigma_{g,n}$. 
The link $\partial F$ (usually together with $F$) is called a $homologically\ fibered\ link$ in $M$ and we call $F$ a $homologially\ fibered\ surface$. \\
By the above observations, for a link with a Seifert surface $F$ in a closed 3-manifold $M$ 
being a homologically fibered link with homologically fibered surface, it is enough to check that push-ups (or push-downs) of simple closed curves on $F$ which form a basis of $H_{1}(F)$ also form a basis of $H_{1}(M\setminus Int(F\times[0,1]))$.

 \subsection{Mixing diagrams}
$L(p,q)$ is obtained by the $\left( -\frac{p}{q}\right)$--Dehn surgery along the unknot $U$ in $S^3$. 
In $L(p,q)$, we see the knot $U'$, which is the image of $U$ under the surgery. 
For any graph (1--dimensional subcomplex of a trianglation of $L(p,q)$) $G'$, 
we can isotope $G'$ so that this is disjoint from $U'$. 
Therefore we get a graph $G$ in $S^3$, which becomes $G'$ after the surgery. 
We call $U \cup G$ in $S^3$ a $\it {mixing \ diagram}$ of $G'$ in $L(p,q)$.

\section{Proof of Proposition \ref{prop}} \label{sec3}
We fix a lens space $L(p,q)$. 
Suppose there exists a homologically fibered link in $L(p,q)$ with a Seifert surface $F$ which is homomorphic to $\Sigma_{0,n+1}$. 
We take a spine of $F$, $S'=K'_{1} \cup v'_{1} \cup K'_{2} \cup \cdots \cup v'_{n-1} \cup K'_{n}$ as in Figure \ref{fig:one}. When $n=1$, there are no $v'$-arcs. 
Note that $\{K'_{1}, K'_{2},\ldots ,K'_{n}\}$ forms a basis of $H_{1}(F)$. 
Accurately, $F$ is knotted $\Sigma_{0,n+1}$ in $L(p,q)$, thus we fix an embedding. 
Cutting $L(p,q)$ along $F$ is equivalent to remove $N(S)$: a small neighborhood of $S$, except for the information about a partition of the boundary. 
By choosing a longitude of $K'_i$ and a twist number of the band corresponding to $v'_{j}$, we can assign the information about a partition of the boundary. 

\begin{figure}[htbp]
 \begin{center}
  \includegraphics[width=100mm]{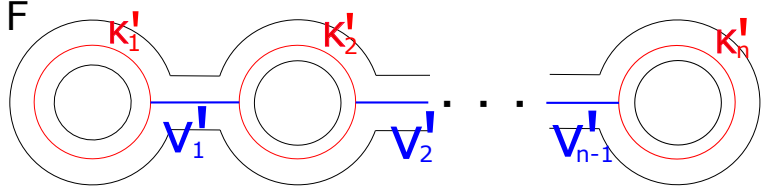}
 \end{center}
 \caption{$S'$, a spine of $F$}
 \label{fig:one}
\end{figure}

We take a mixing diagram $U\cup S$ of $S'$. 
We denote by $K_{i}$ the loop of $S$ corresponding to $K'_{i}$, and by $v_{i}$ the arc corresponding to $v'_{i}$. 
We take meridian curves of $N(S)$, $m_{i}$ which is dual to $K_{i}$, and $x_{i}$ which is dual to $v_{i}$ as in Figure \ref{fig:two}. 
Note that $x_i$ bounds a surface homeomorphic to $\Sigma_{i,1}$ in $S^3\setminus (N(U)\cup N(S))$, which lies on the boundary. 
We also take the meridian curve $m$ of $U$. 

\begin{figure}[htbp]
 \begin{center}
  \includegraphics[width=100mm]{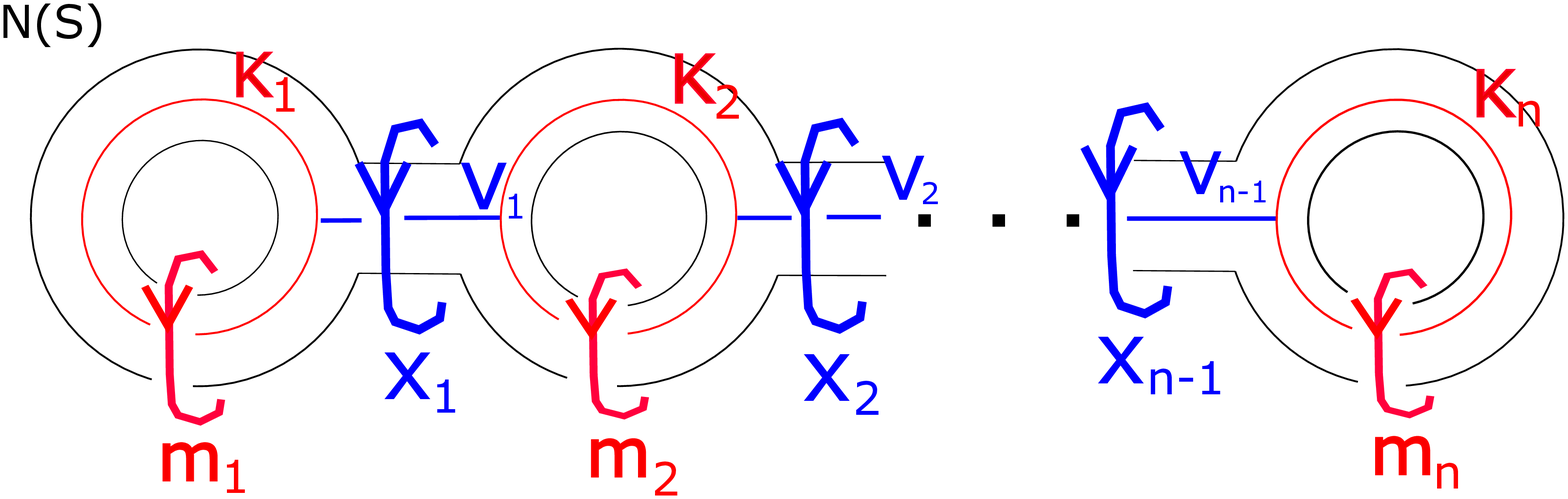}
 \end{center}
 \caption{meridian curves}
 \label{fig:two}
\end{figure}

Since $x_{i}$ is null-homologous in $S^3 \setminus (N(U) \cup N(S))$, by using Mayer--Vietoris exact sequence, we can see $H_{1}(S^3 \setminus (N(U) \cup N(S)))$ is isomorphic to $H_{1}(S^3 \setminus (N(U)\cup N(\cup^{n}_{i=1}K_i )))$, which is isomorphic to $\mathbb{Z}^{n+1}$ and generated by $\{m,m_1,\ldots ,m_n\}$. 
After $\left( -\frac{p}{q} \right)$--surgery along $U$, we see $H_{1}(L(p,q)\setminus N(S'))$ is isomorphic to $H_{1}(L(p,q) \setminus N(\cup^{n}_{i=1}K'_i))$. \\

We will compute $H_{1}(L(p,q) \setminus N(\cup^{n}_{i=1}K'_i))$ from $H_{1}(S^3 \setminus (N(U)\cup N(\cup^{n}_{i=1}K_i )))$. 
We denote by $a_{i}$ the linking number of $U$ and $K_{i}$. 
Then, the canonical longitude of $U$ is represented by $\Sigma^{n}_{i=1} a_{i}m_{i}$ in $H_{1}(S^3 \setminus (N(U)\cup N(\cup^{n}_{i=1}K_i )))$. 
Therefore, $H_{1}(L(p,q) \setminus N(\cup^{n}_{i=1}K'_i))$ is isomorphic to $\langle m,m_1,\cdots ,m_n \rangle /(pm-q\Sigma^{n}_{i=1} a_{i}m_{i})$, which is isomorphic to $\mathbb{Z}^{n}\bigoplus \mathbb{Z}/\gcd(p,qa_1,\cdots ,qa_n)\mathbb{Z}$, where $\gcd$ represents the greatest common divisor. 
We fix this presentation of $H_{1}(L(p,q) \setminus N(\cup^{n}_{i=1}K'_i))$.\\

Since $L(p,q\setminus (F\times[0,1]))$ is a homology cobordism over $\Sigma_{0,n+1}$, $H_{1}(L(p,q)\setminus (F\times [0,1])) \cong H_{1}(L(p,q)\setminus N(S')) \cong H_{1}(L(p,q) \setminus N(\cup^{n}_{i=1}K'_i))$ must be isomorphic to $\mathbb{Z}^n$. 
Hence, $\gcd(p,qa_1,\cdots ,qa_n)=1$.\\

Moreover, being a homology cobordism over $\Sigma_{0,n+1}$, there must be a longitude $\bar{K'_{i}}$ for each $K'_{i}$ such that $\{\bar{K'_{1}},\cdots ,\bar{K'_n}\}$ forms a basis of $H_{1}(L(p,q)\setminus N(S'))$. 
In $S^3$, the longitude $\bar{K'_{i}}$ corresponds to a longitude $\bar{K_i}$ for $K_i$. 
By setting $l_{i,j}$ for $i\neq j$ be the linking number of $K_i$ and $K_{j}$ ($l_{i,j}=l_{j,i}$), 
this $\bar{K_i}$ is represented by using some integer $t_{i}$ as $t_{i}m_{i}+a_{i}m+\Sigma_{j\neq i}l_{i,j}m_{j}$ in $H_{1}(S^3\setminus (N(U)\cup N(\cup^{n}_{i=1}K_i ))) \cong \mathbb{Z}^{n+1}$. 
For 
\begin{eqnarray*}
\{pm-q\Sigma^{n}_{i=1}a_{i}m_i, t_{1}m_{1}+a_{1}m+\Sigma_{j\neq 1}l_{1,j}m_{j}, t_{2}m_{2}+a_{2}m+\Sigma_{j\neq 2}l_{2,j}m_{j},\cdots ,t_{n}m_{n}+a_{n}m+\Sigma_{j\neq n}l_{n,j}m_{j}\}
\end{eqnarray*}
  forming a basis of $H_{1}(L(p,q) \setminus N(\cup^{n}_{i=1}K'_i)) \cong \mathbb{Z}^{n}\bigoplus \mathbb{Z}/\gcd(p,qa_1,\cdots ,qa_n)\mathbb{Z}$, the following equation must be satisfied:\\

\begin{eqnarray}
\left|
\begin{array}{ccccc}
   p      &   -qa_{1}  &   -qa_{2}   & \cdots &   -qa_{n}\\
   a_{1}  &     t_{1}    &   l_{1,2}    &  \cdots &   l_{1,n} \\
   a_{2}  &    l_{2,1}   &    t_{2}     & \cdots  &   l_{2,n}\\
\vdots &  \vdots    &   \vdots  &  \ddots & \vdots \\
   a_{n}  &   l_{n,1}    &     l_{n,2}  &  \cdots &     t_{n}\\
\end{array}
\right| = \pm 1 \nonumber   
\end{eqnarray}

By using the relations $l_{i,j}=l_{j,i}$ for $i\neq j$, we obtain equation (1). \\

Conversely, if we have integers $a_{h}$, $l_{i,j}$ and $t_k$($h,i,j,k=1,2,\dots,n$, $i\lneq j$) satisfying the equation (1), 
we can construct framed knots $K_{1}, K_{2},\cdots ,K_{n}$ in $S^3$ which are disjoint from $U$ satisfying the following:\\
\begin{itemize}
 \item the linking number of $U$ and $K_i$ is $a_{i}$.
 \item the linking number of $K_i$ and $K_j$ for $i\lneq j$ is $l_{i,j}$.
 \item the framing of $K_i$ is the $t_i$-times twisting of the canonical longitude of $K_i$ along the meridian disk.
\end{itemize}
Reversing the above computation, we can construct a homologically fibered link with Seifert  surface which is homeomorphic to $\Sigma_{0,n+1}$. \\

\section{Proof of theorems } \label{sec4}
\subsection{Proof of Theorem \ref{thm1}}
 
By Proposition \ref{prop}, that $L(p,q)$ is realizable as a closure of a homology cobordism over $\Sigma_{0,2}$ is equivalent to that there exist integers $a$ and $t$ satisfying 
$
\left|
\begin{array}{cc}
   p      &   -qa \\
   a &     t\\
\end{array}
\right| = \pm 1
$. 
This equation is equivalent to $tp+qa^{2}=\pm 1$. Note that this implies $p$ and $a^2$ are coprime and therefore $p$ and $a$ are coprime. 
This $a$ satisfies $qa^{2} \equiv \pm 1$ mod $p$. 
Since $p$ and $a$ are coprime, this is equivalent to $q \equiv \pm {(a^{-1})}^2$  mod $p$. 
Therefore the condition for $L(p,q)$ being realizable as a closure of a homology cobordism over $\Sigma_{0,2}$ is equivalent to that $q$ or $-q$ is quadratic residue modulo $p$.

\subsection{Proof of Theorem \ref{thm2}}
We fix a lens space $L(p,q)$ with $p,q > 0$ (every lens space has such a representation). 
We fix a positive integers $s$ and $r$ such that $ps-qr=1$. \\\\

By Proposition \ref{prop}, that $L(p,q)$ is realizable as a closure of a homology cobordism over $\Sigma_{0,2}$ is equivalent to that there exist integers $a_1,a_2,l_{1,2},t_1$, and $t_2$ satisfying 
$
\left|
\begin{array}{ccc}
   p      &   -qa_1  &   -qa_2 \\
   a_1 &     t_1    &    l_{1,2}\\
   a_2 &     l_{1,2}&     t_2  \\
\end{array}
\right| = \pm 1
$. 
This equation is equivalent to $p(t_{1}t_{2}-l^{2}_{1,2})-q(2l_{1,2}a_{1}a_{2}-t_{2}a^{2}_{1}-t_{1}a^{2}_{2})=\pm1$, and also equivalent to 
 \begin{eqnarray}
  \begin{cases}
  t_{2}a^2_{1}-2l_{1,2}a_{1}a_{2}+t_{1}a^{2}_{2}=-\epsilon (r+kp) &\\
  \left|
  \begin{array}{cc}
   t_2       &  -l_{1,2}  \\
   -l_{1,2}  &    t_{1}   \\
  \end{array}
  \right| = \epsilon (s+kq) & 
 
  \end{cases}
 \end{eqnarray}\\

We use the following lemma (see Section 5.3 of \cite{quadra} for example). 
Here the determinant of $ax^{2}+2bxy+cy^{2}$ is $ac-b^2$.
\begin{lem} \label{lem1}
For $n,D\in \mathbb{Z}$, there is a binary quadratic form $f(x,y)\in \mathbb{Z}[x,y]$ with determinant $D$ such that $f(x,y)=n$ has a primitive solution if and only if the congruence equation $-z^{2} \equiv D$ mod $n$ has a solution. 
\end{lem}

\begin{proof}
(if part)
We set $z_{0}, C_{0}\in \mathbb{Z}$ to be $D=-z^{2}_{0}+C_{0}n$. 
Then $f_{0}(x,y)=nx^{2}+2z_{0}xy+C_{0}y^2$, whose determinant is $D$ satisfies $f_{0}(1,0)=n$.\\
(only if part)
Suppose $f(x,y)=Ax^{2}+2Bxy+Cy^{2}$ has primitive solution $(x,y)=(x_0,y_0)$ for $f(x,y)=n$. 
Fix integers $\bar{x}, \bar{y}$ such that $x_{0}\bar{x}-y_{0}\bar{y}=1$. 
Then $g(x,y)=f\left( U\left( \begin{array}{c} x\\ y \end{array} \right) \right) =f(x_0,y_0)x^{2}+2B'xy+C'y^2$, where $U=\left( \begin{array}{cc} x_0 & \bar{y} \\ y_0 & \bar{x} \end{array} \right)$, and $B', C'$ are some integers, has the same determinant $D$ as $f(x,y)$. 
Thus $D=C'n-B'^{2}$. This implies $-B'^{2} \equiv D$ mod $n$. 
\end{proof}

By the above lemma, the existence of a solution of (2) is equivalent to the existence of integers $k$ and $z'$, non-zero integer $w$ and $\epsilon {'}$, which is $+1$ or $-1$  satisfying the following equation:
\begin{eqnarray*}
 \begin{cases}
-z'^2 \equiv \epsilon' (s+kq) \ \ \ \rm{mod} \ \ -\epsilon' \frac{r+kp}{w^2} &\\
w^{2}\mid r+kp&
 \end{cases}
\end{eqnarray*}
Note that since $(s+kq)$ and $(r+kp)$ are coprime, $z'^2$ and therefore $z'$ are prime to $(r+kp)$. 
This equation is equivalent to the following
 by setting $z=z'^{-1}$ and $\epsilon = \epsilon'$:
\begin{eqnarray}
 \begin{cases}
p \equiv \epsilon z^2 \ \ \ \rm{mod} \ \ \frac{r+kp}{w^2} &\\
w^{2}\mid r+kp& 
 \end{cases}
\end{eqnarray}
\\\\
Since $L(p,q) \cong L(p,r)$, we can replace $r$ in (3) with $q$. \\
We will show that for all $(p,q)$, there exist integers satisfying (3) by using the following two facts, which and whose use are fully explained in Section $3$ of \cite{nozaki}:
\\
Put $K_1 = \mathbb{Q}(\zeta_p)$ and $K_2 = \mathbb{Q}(\sqrt{-1})$, where $\zeta_p = \exp(\frac{2\pi}{p} \sqrt{-1})$.

\begin{fact} \label{fact1}
(A special case of Chebotarev density theorem, for example Theorem $10$ of \rm{ \cite{density}} )\\
If there exists $\sigma \in Gal(K_{1}K_{2}/\mathbb{Q})$ satisfying:\\
\begin{itemize}
 \item $\sigma |_{K_1}$ is $[\zeta_{p} \longmapsto \zeta^{m}_{p}]$, where $m$ is prime to $p$.
 \item $\sigma |_{K_2}$ is $[\sqrt{-1} \longmapsto -\sqrt{-1}]$.
\end{itemize}
then there exist infinitely many integers $k$ satisfying:\\
\begin{itemize}
 \item $m+kp$ is prime.\\
 \item $m+kp \equiv -1$ mod $4$
\end{itemize}
\end{fact}

\begin{fact} \label{fact2}
(See Corollary $4.5.4$ of {\rm{ \cite{cycle}} }for example)\ 
$\sqrt{-1} \in \mathbb{Q}(\zeta_{p})$ if and only if $4\mid p$.
\end{fact}
\vspace{0.5cm}
From these we have the following:
\begin{lem} \label{make_prime}
There exists an integer $k$ satisfying at least one of the following :\\
\begin{itemize}
 \item $q+kp$ is prime and $q+kp \equiv -1$ \rm{mod} $4$.\\
 \item $r+kp$ \it{is prime and} $r+kp \equiv -1$ \rm{mod} $4$.
\end{itemize}
\end{lem}

\begin{proof}
(i) the case when $4\nmid p$\\
In this case, $K_{1}\cap K_{2}=\mathbb{Q}$ by Fact \ref{fact2}. 
Thus in $Gal(K_{1}K_{2}/\mathbb{Q})$, $[\zeta_{p} \longmapsto \zeta^{q}_{p}]$ and $[\sqrt{-1} \longmapsto -\sqrt{-1}]$ are coexistable. 
By Fact \ref{fact1}, we have an integer $k$ such that $q+kp$ is prime and $q+kp \equiv -1$ mod $4$.\\
(ii) the case when $4 \mid p$.\\
We write $p=4p'$. In this case, $K_{1}K_{2}=K_{1}$ by Fact \ref{fact2}. 
Note that $\zeta^{p'}_p = \sqrt{-1}$. 
If $\zeta_p$ is mapped to $\zeta^{m}_p$ ($m$ is $q$ or $r$), then $\sqrt{-1}$ is mapped to $\zeta^{mp'}_p$, this is $\sqrt{-1}$ when $mp' \equiv p'$ mod $p$ and $-\sqrt{-1}$ when $mp' \equiv -p'$ mod $p$. 
Note that $mp' \equiv \pm p'$ mod $p$ is equivalent to $m \equiv \pm 1$ mod $4$. 
Since $ps-qr=1$, one of $q$ and $r$ is congruent to $-1$ modulo $4$. 
Therefore, one of $[\zeta_{p} \longmapsto \zeta^{q}_{p}]$ and $[\zeta_{p} \longmapsto \zeta^{r}_{p}]$ maps $[\sqrt{-1} \longmapsto -\sqrt{-1}]$. 
By Fact \ref{fact1}, there exists an integer $k$ s.t. 
$q+kp$ is prime and $q+kp \equiv -1$ mod $4$, or $r+kp$ is prime and $r+kp \equiv -1$ mod $4$.
 
\end{proof}

For $L(p,q)$, there exists an integer $q'$, which is $q+kp$ or $r+kp$ s.t. $L(p,q) \cong L(p,q')$, $q'$ is prime, and $q'  \equiv -1$ mod $4$ by Lemma \ref{make_prime}. 
By using the Legendre symbol and Euler's criterion,\\
$\left( \frac{-p}{q'} \right) \equiv \left( \frac{-1}{q'} \right) \left( \frac{p}{q'} \right) \equiv (-1)^{\frac{q'-1}{2}} \left( \frac{p}{q'} \right) \equiv -\left( \frac{p}{q'} \right) $ mod $q'$\\
This implies there exists $\epsilon \in \{\pm1\}$ such that $\left( \frac{\epsilon p}{q'} \right) =1$, i.e. 
$\epsilon p$ is quadratic residue modulo $q'$. 
Therefore, the condition (3) is satisfied.

\section{Applications for hc($\cdot$)} \label{sec5}
\begin{defini} \rm{(Remark 6.10 in \cite{hc})}\ 
For a connected orientable closed 3-manifold $M$, the integer hc($M$) is defined as the minimal number $g$ such that $M$ is representable as a closure of a homology cobordism over $\Sigma_{g,1}$.
\end{defini}

In \cite{nozaki}, it was proved that for any 3-manifold $M$ whose first homology group is non-trivial and generated by one element, hc($M$)$=1$. 
However in general for a 3-manifold $M$ whose first homology group has a non-trivial torsion, the computation of hc($M$) becomes difficult.\\
Since it is known that under plumbing operations, pairs of a 3-manifold and a homologically fibered link with Seifert surface in it are closed, we get some results about hc($\cdot$) as follows.

\begin{cor}
(1)\ $\rm{hc}\left( \it{L(p_1,q_1) \# L(p_2,q_2)}\right) \leq 1$ if $q_1$ is a quadratic residue modulo $p_1$ and $q_2$ is a quadratic residue modulo $p_2$.\\
(2)\ $\rm{hc}\left( \it{L(p,q) \# L(p_1,q_1) \# L(p_2,q_2)}\right) \leq 2$ if $q_1$ is a quadratic residue modulo $p_1$ and $q_2$ is a quadratic residue modulo $p_2$.
\end{cor}

\vspace{0.5cm}

\ GRADUATE SCHOOL OF MATHEMATICAL SCIENCES, THE UNIVERSITY OF TOKYO, 3-8-1 KOMABA, MEGURO--KU, TOKYO, 153-8914, JAPAN\\
\ \ E-mail address: \texttt{sekino@ms.u-tokyo.ac.jp}

\end{document}